\documentclass[psamsfonts]{amsart}

\pdfoutput=1

\usepackage[dvipsnames]{xcolor}

\usepackage{amsmath}
\usepackage{amssymb}
\usepackage{amsthm}
\usepackage{amsrefs}
\usepackage{amsfonts}
\usepackage{mathrsfs}
\usepackage{mathtools}
\usepackage{stmaryrd}
\usepackage[mathcal]{eucal}
\usepackage[all]{xy}
\usepackage{verbatim}  
\usepackage{multicol} 
\usepackage{enumerate}
\usepackage{nicefrac}
\usepackage{bbm}

\newcommand{\bbC}{\mathbb{C}}

\newcommand{\bbF}{\mathbb{F}}
\newcommand{\bbG}{\mathbb{G}}

\newcommand{\bbQ}{\mathbb{Q}}
\newcommand{\bbR}{\mathbb{R}}

\newcommand{\bbZ}{\mathbb{Z}}

\newcommand{\cA}{\mathcal{A}}

\newcommand{\cC}{\mathcal{C}}

\newcommand{\cG}{\mathcal{G}}

\newcommand{\cO}{\mathcal{O}}

\newcommand{\cT}{\mathcal{T}}

\newcommand{\fa}{\mathfrak{a}}

\newcommand{\fg}{\mathfrak{g}}
\newcommand{\fh}{\mathfrak{h}}

\newcommand{\fk}{\mathfrak{k}}

\newcommand{\fp}{\mathfrak{p}}

\renewcommand{\phi}{\varphi}

\providecommand{\abs}[1]{\left\lvert #1 \right\rvert}

\providecommand{\ar}{\rightarrow}

\DeclareMathOperator{\ad}{ad}

\DeclareMathOperator{\Ind}{Ind}
\DeclareMathOperator{\Res}{Res}
\DeclareMathOperator{\rk}{rk}
\DeclareMathOperator{\cd}{cd}

\providecommand{\tensor}{\otimes}

\DeclareMathOperator{\Hom}{Hom}

\DeclareMathOperator{\GL}{GL}
\DeclareMathOperator{\SL}{SL}
\DeclareMathOperator{\Sp}{Sp}
\DeclareMathOperator{\PGL}{PGL}
\DeclareMathOperator{\PSL}{PSL}

\DeclareMathOperator{\SU}{SU}
\DeclareMathOperator{\SO}{SO}

\DeclareMathOperator{\HH}{H}

\DeclareMathOperator{\St}{St}

\DeclareMathOperator{\Span}{span}

\newtheorem{thm}{Theorem}[section]
\newtheorem{cor}[thm]{Corollary}
\newtheorem{prop}[thm]{Proposition}
\newtheorem*{prop*}{Proposition}
\newtheorem{lem}[thm]{Lemma}

\newtheorem{mainthm}{Theorem}

\newtheorem{maincor}[mainthm]{Corollary}

\theoremstyle{definition}
\newtheorem{defn}[thm]{Definition}

\newtheorem{exmp}[thm]{Example}

\numberwithin{claim}{thm}

\theoremstyle{remark}
\newtheorem{rem}[thm]{Remark}

\makeatletter
\let\c@equation\c@thm
\makeatother
\numberwithin{equation}{section}

\usepackage{hyperref}
\hypersetup{
	colorlinks = true,
	citecolor=Green,
	linkcolor=blue,
	urlcolor=violet,
}

\title[High-deg. cohomology of congruence subgroups via $S$-arith. groups]{High-degree cohomology of congruence subgroups of $\SL_n(\mathcal{O})$ via cohomology of $S$-arithmetic groups}

\author{Matthew Scalamandre}
\date{}

\begin{document}
	\begin{abstract}
		If $\mathfrak{p}$ is a prime ideal of a number ring $\mathcal{O}$, then the top-degree cohomology of the principal congruence subgroup of level $\mathfrak{p}$ is naturally a representation of $\SL_n(\mathcal{O}/\mathfrak{p}).$ We prove that the multiplicity of the Steinberg representation in this cohomology space is one. When $\mathcal{O}$ is Euclidean and $\mathfrak{p}$ is suitably small -- for example a universal side divisor -- then we prove that the multiplicity of the Steinberg representation in the next-highest-degree cohomology space is zero. Our proof relies on a computation of the cohomology of an $S$-arithmetic group ouside of a linear range of degrees, derived from work of Blasius--Franke--Grunewald.
	\end{abstract}
	\maketitle
	\tableofcontents

\section{Introduction}

Arithmetic groups such as $\SL_n(\bbZ)$ are of central importance in number theory, group theory, topology, and many other areas. Determining the cohomology of these groups is therefore a fundamental problem. The cohomology of these groups, at least with rational coefficients, is often well-understood in cohomological degrees which are small relative to $n$, by stability results such as \cite{BorelStable}. However, the cohomology in high degrees is much more difficult to compute. 

If $F$ is a number field and $\cO$ is its ring of integers, we recall the following family of arithmetic subgroups of $\SL_n(\cO).$
\begin{defn}
    Let $I$ be an ideal of a number ring $\cO$, and let $n\geq 2$. Then, the \emph{principal congruence subgroup of level} $I$ is \[\Gamma_n(I) = \ker\{\SL_n(\cO)\to \SL_n(\cO/I)\}.\]
\end{defn}

The outer action of the finite group $\SL_n(\cO/I)$ on $\Gamma_n(I)$ induces a natural representation of $\SL_n(\cO/I)$ on each cohomology group $\HH^{q}(\Gamma_n(I);\bbC)$. This structure encodes information about the cohomology of other arithmetic subgroups of $\SL_n(\cO)$. Indeed, we say that a finite-index subgroup $\Gamma\leq \SL_n(\cO_F)$ is a \emph{congruence subgroup} if for some $I$, we have $\Gamma_n(I)\leq \Gamma$. If $F$ is not totally imaginary, then every arithmetic subgroup has this form by a theorem of Bass--Milnor--Serre \cite{BassMilnorSerre}. Since $\Gamma_n(I)$ is normal in $\SL_n(\cO)$, it is normal in $\Gamma,$ and so $\Gamma/\Gamma_n(I)$ is a finite group. By the Hochschild--Serre spectral sequence, we have:
\[\HH^q(\Gamma;\bbC)\cong \HH^q(\Gamma_n(I);\bbC)^{\Gamma/\Gamma_n(I)}.\]
Our primary goal will be to investigate this structure when $I=\fp$ is a prime ideal.

\subsection{Steinberg representations and high-degree cohomology}

In \cite{BS}, Borel and Serre prove that arithmetic groups are \emph{rational Bieri-Eckmann duality groups}; i.e., for any such group $\Gamma$, there is an integer $\nu$ and a $\bbQ[\Gamma]$-module $D$ such that \[\HH^{q}(\Gamma;\bbQ)\cong \HH_{\nu-q}(\Gamma;D).\] The integer $\nu$ is necessarily the \emph{rational cohomological dimension} of $\Gamma$. In particular, if $M$ is any $\bbQ[\Gamma]$-module and $q>\nu,$ then \[\HH^q(\Gamma;M) = 0.\]

Let $\Gamma$ be a finite-index subgroup of $\SL_n(\cO)$, where $\cO$ is the ring of integers of a number field $F$. We say that $F$ has \emph{signature} $(r,s)$ if $F\tensor_\bbQ \bbR \cong \bbR^r\oplus \bbC^s.$ For these arithmetic groups $\Gamma$, Borel--Serre compute that $\cd_\bbQ \Gamma = \cd_\bbQ\SL_n(\cO) = \nu_F(n)$, where
\[\nu_F(n) = r\cdot\binom{n+1}{2} + s\cdot (n^2-1) -n+1.\] They likewise show that $D$ is the \emph{Steinberg representation} $\St_n(F)$. This is the top-degree rational homology of the \emph{Tits building} $\cT_n(F),$ which is the $(n-2)$-dimensional simplicial complex whose $k$-simplices are flags of linear subspaces
\[0\subsetneq V_0 \subsetneq \dots \subsetneq V_k \subsetneq F^n.\] 

This definition makes sense for any field, and in fact the representation $\St_n(L)$ is irreducible for any field $L$ \cites{Steinberg, SteinIrred}. In particular, if $\fp$ is a prime ideal of a number ring $\cO$, then $\St_n(\cO/\fp)$ is an irreducible $\SL_n(\cO/\fp)$-representation.

Using Bieri-Eckmann duality, one produces a natural map of representations:
\[\HH^\nu(\Gamma_n(\fp);\bbC)\cong (\St_n(F))_{\Gamma_n(\fp)}\to \St_n(\cO/\fp).\]
The Solomon--Tits theorem \cite{Solomon} gives a generating set for $\St_n(\cO/\fp)$, from which it follows that this map is surjective. Since the category $\bbC[\SL_n(\cO/\fp)]$ is semisimple, there is a copy of the Steinberg representation $\St_n(\cO/\fp)$ in the direct sum decomposition of $\HH^\nu(\Gamma_n(\fp);\bbC)$ into irreducible representations.

When $\cO$ is a Euclidean domain and $\fp = p\cO$ for $p$ sufficiently small\footnote{See Pal \cite[Theorem 1.5]{Pal} for an exact characterization of ``sufficiently small". When $\cO = \bbZ$, the earlier work of Miller--Patzt--Putman shows this isomorphism for $p\leq 5$. They prove this range is sharp.}, results of Miller--Patzt--Putman \cite{MPP} and Pal \cite{Pal} imply an exact description of the representation $\HH^\nu(\Gamma_n(\fp);\bbC).$ They prove that \[ \HH^\nu(\Gamma_n(\fp);\bbC)\cong \tilde{\HH}_{n-2}(\cT_n(F)/\Gamma_n(\fp);\bbC). \] 
We refer to the latter representation as $\St^{\cO^\times}_n(\cO/\fp)$. It is possible to analyze the structure of this representation, and prove that the Steinberg representation appears with multiplicity one.

Our main theorem states that this multiplicity result is valid in nearly all cases, removing the hypotheses of Miller--Patzt--Putman and Pal. For any number field $F$ and any prime ideal $\fp$, if $n$ is not too small, then the Steinberg representation appears exactly once in the top-degree cohomology of the principal congruence subgroup. 
\begin{mainthm}\label{mainthm:multiplicity}
    Let $\cO$ be the ring of integers of a number field $F$, and let $\fp$ be a prime ideal. Let $\nu$ be the rational cohomological dimension of $\SL_n(\cO_F).$ Let \[n\geq \begin{cases}
        4 & F = \bbQ \\ 3 & F \text{ imaginary quadratic} \\ 2 & \text{otherwise.}
    \end{cases} \]Then, the $\SL_n(\cO/\fp)$-representation $\St_n(\cO/\fp)$ appears with multiplicity one in the direct sum decomposition of $\HH^\nu(\Gamma_n(\fp);\bbC).$ 
\end{mainthm}
When $F=\bbQ$ and $n=2,$ it is not hard to deduce that the multiplicity bound appearing in Theorem \ref{mainthm:multiplicity} does not hold for sufficiently large primes, using explicit classical computations of these cohomology spaces. Therefore, the lower bound appearing in the theorem is necessary. It is not obvious to the author that these bounds, or the bounds appearing in Theorem \ref{mainthm:multiplicityCodimOne} below, are optimal in all cases.

In the cases where $\HH^\nu(\Gamma_n(\fp);\bbC)$ is known, we can extend our methods to produce information in codimension 1:
\begin{mainthm}\label{mainthm:multiplicityCodimOne}
    Let $\cO$ be the ring of integers of a number field $F$, let $\fp$ be a prime ideal, and let $\nu$ be the rational cohomological dimension of $\SL_n(\cO).$ Assume that \[ \HH^\nu(\Gamma_n(\fp);\bbC)\cong \St^{\cO^\times}_n(\cO/\fp). \] Let \[ n\geq \begin{cases} 5 & F =\bbQ \\ 4 & F \text{ imaginary quadratic} \\3 & F \text{ real quadratic, or has signature }(1,1) \\ 2 & \text{otherwise.} \end{cases} \] Then, the $\SL_n(\cO/\fp)$-representation $\St_n(\cO/\fp)$ does not appear in the direct sum decomposition of $\HH^{\nu-1}(\Gamma_n(\fp);\bbC).$ 
\end{mainthm}

In general, an explicit description of the cohomology spaces appearing in Theorems \ref{mainthm:multiplicity} and \ref{mainthm:multiplicityCodimOne} is not known, and we do not produce one. Our methods are indirect, and involve the cohomology of $S$-arithmetic groups.

\subsection{Cohomology of $S$-arithmetic groups} We first recall the rings of $S$-integers:

\begin{defn}
	Let $F$ be a number field. Let $\Omega$ be the set of places of $F$, thought of as equivalence classes of nontrivial valuations on $F$. The set of \emph{infinite places} $\Omega_\infty$ consist of the Archimedian valuations corresponding to embeddings of $F$ into $\bbR$ or $\bbC$. If $F$ has signature $(r,s)$, then there are $r$ places corresponding to embeddings into $\bbR$ and $s$ places corresponding to conjugate pairs of embeddings into $\bbC$. The remaining places, which correspond to non-Archimedian valuations, form the set $\Omega_f$ of \emph{finite places} of $F$. These correspond to embeddings of $F$ into $p$-adic fields. Each finite place is equivalent to $\abs{-}_{\fp}$ for exactly one prime ideal $\fp$ of $\cO_F$. We say $S$ is a \emph{set of places} if we have $S = S_f\cup \Omega_\infty$, where $S_f$ is a finite subset of $\Omega_f$. We define the \emph{ring of $S$-integers}:
	$$\cO_{F,S} = \{ x\in F \mid \abs{x}_v\leq 1 \text{ for all } v\in \Omega\setminus S\}.$$
	When $S=\Omega_\infty$, this is just the ring of integers $\cO_F$.
\end{defn}

A prototypical example is the following. Let $m\in\bbZ$, and let \[S_m = \{\abs{-}_p\mid p \text{ divides } m\}\cup \Omega_\infty.\] Then, we have  $\cO_{\bbQ,S_m}=\bbZ[\frac{1}{m}]$. We will be interested in the groups $\SL_n(\cO_{F,S}).$ Borel and Serre have likewise proved that $\SL_n(\cO_{F,S})$ is a duality group and computed the rational cohomological dimension \cite{BoSe2}. In particular:
\[\cd_{\bbQ}{\SL_n(\cO_{F,S})} = \nu_F(n) + \abs{S_f} (n-1).\]

We can now explain the main ingredients of the proofs of Theorems \ref{mainthm:multiplicity} and \ref{mainthm:multiplicityCodimOne}. The first, which is the core topological and algebraic input, involves the construction and analysis of a spectral sequence (cf. Theorem \ref{thm:SpectralSequence}). If $S_\fp = \{\abs{-}_\fp\}\cup \Omega_\infty$, this is an $E_1$-spectral sequence converging to the cohomology of $\SL_n(\cO_{F,S_\fp})$, such that 
    \begin{enumerate}
        \item $E_1^{p,q}$ can be described in terms of the cohomology of arithmetic subgroups of $\SL_n(\cO_F).$
        \item $\dim E_2^{n-1,q}$ is an upper bound on the multiplicity of the Steinberg module in $\HH^q(\Gamma_n(\fp);\bbC).$
    \end{enumerate}

The second ingredient, which explains the bounds on $n$ appearing in Theorems \ref{mainthm:multiplicity} and \ref{mainthm:multiplicityCodimOne}, is the use of the automorphic decomposition to compute the cohomology of $\SL_n(\cO_{F,S})$ outside a linear range around the middle dimension. The following result is obtained quickly from the main theorem of Blasius--Franke--Grunewald \cite{BFG}. 

\begin{mainthm}\label{mainthm:automorphic}
    Let $F$ be a number field of signature $(r,s)$, let $S_f$ be a nonempty set of finite places, and let $S = S_f\cup \Omega_\infty$. Let $\Gamma$ be a finite index subgroup of $\SL_n(\cO_{F,S})$. Let $M_{r,s}$ be the compact symmetric space \[\prod_{i=1}^r (\SU_n/\SO_n) \times \prod_{i=1}^s \SU_n.\] Then, for 
    
    \begin{align*} 
    q\not\in \left[\tfrac{r\left(\binom{n+1}{2}-\lceil\frac{n}{2}\rceil +1\right) +s(n^2-n)}{2}\right. +  \abs{S_f}&(n-1),  \\  &\left.\tfrac{r\left(\binom{n+1}{2} +\lceil\frac{n}{2}\rceil -1\right) +s(n^2+n-2) }{2} + \abs{S_f}(n-1)\right],
    \end{align*}
    we have $\HH^q(\Gamma;\bbC)\cong \HH^q(M_{r,s};\bbC)$
\end{mainthm}

Blasius--Franke--Grunewald use the Franke filtration \cite{F} of the space of automorphic forms to show that when $S\ne \Omega_\infty,$ the cohomology of an $S$-arithmetic group $\Gamma$ agrees with its $L^2$-cohomology. They, along with Borel--Labesse--Schwermer \cite{BLS}, observe that a subrepresentation of $L^2(\Gamma\backslash G)$ contributes to cohomology only if it is either trivial or cuspidal. 

The contribution of the trivial representation to the cohomology of $\Gamma$ is well-known to agree with the cohomology of $M_{r,s}$ \cite{BorelStable}. The cuspidal cohomology is much more difficult to compute, and we do not attempt to do so. However, for arithmetic subgroups of $\SL_n(\cO_F)$, it is known that the cuspidal cohomology vanishes outside of a linear range around the middle dimension. See \cite{SchwermerHolomorphy} for a precise statement. To prove Theorem \ref{mainthm:automorphic}, we verify that a similar result also holds for finite-index subgroups of $\SL_n(\cO_{F,S})$. 

If $\nu = \cd_{\bbQ} \SL_n(\cO_{F,S})$ and \[i< r\left(\tfrac{\binom{n+1}{2} -  \lceil\frac{n}{2}\rceil}{2}\right)+s\left(\frac{n^2-n}{2}\right)-n+1,\] then there is no cuspidal cohomology in $\HH^{\nu-i}(\SL_n(\cO_{F,S});\bbC)$. The bounds appearing in Theorems \ref{mainthm:multiplicity} and \ref{mainthm:multiplicityCodimOne} are derived from this inequality. 

Notice that as $\abs{S_f}$ increases, both the upper and lower bounds of the interval in Theorem \ref{mainthm:automorphic} increase by $\rk(\SL_n) = n-1$. This happens at the same rate as $\cd_{\bbQ} \SL_n(\cO_{F,S})$ increases, hence the inequality above does not depend on $\abs{S_f}$. However, the dimension of $M_{r,s}$ is unchanged. When $\abs{S_f}>1,$ this observation implies that cohomology near the rational dimension vanishes. The cohomology does \emph{not} vanish near the rational dimension if $\abs{S_f}=1.$

\begin{maincor}\label{maincor:highdeg}
    Let $F$ be a number field of signature $(r,s)$, and let $S$ be a set of places. Let $\nu $ be the rational dimension of $\SL_n(\cO_{F,S}).$ Let \[0\leq i< r\left(\tfrac{\binom{n+1}{2} -  \lceil\frac{n}{2}\rceil}{2}\right)+s\left(\frac{n^2-n}{2}\right)-n+1.\]
    \begin{enumerate}
        \item If additionally $i<(\abs{S_f} - 1)(n-1),$ then $\HH^{\nu-i}(\SL_n(\cO_{F,S});\bbC) = 0.$
        \item If $i =(\abs{S_f} - 1)(n-1) + j$ for $j\leq 2,$ then
        \[\HH^{\nu - i}(\SL_n(\cO_{F,S});\bbC) = \begin{cases} \bbC & j=0 \\ 0 & j=1 \\ \bbC^s & j=2 \end{cases}\]
    \end{enumerate}
\end{maincor}

\subsection{Organization of the paper} In \S2 we collect some results about buildings, and consider the action of $\SL_n(\cO_{F,S})$ on the Euclidean building associated to $\SL_n(\hat{F}_p).$ Our topological insights provide a spectral sequence converging to the cohomology of $\SL_n(\cO_{F,S})$, and we prove vanishing lines on the $E_1$ page. In \S3, we provide an overview of the automorphic decomposition of the cohomology of arithmetic groups, and use it prove Theorem \ref{mainthm:automorphic}. Corollary \ref{maincor:highdeg} immediately follows. In \S4 we review some representation theory of groups of Lie type, then use our spectral sequence to relate the top-degree cohomology of $\SL_n(\cO_{F,S})$ to the multiplicity of the Steinberg representation in $\HH^{\nu_F}(\Gamma_n(\fp);\bbQ).$ We prove Theorems \ref{mainthm:multiplicity} and \ref{mainthm:multiplicityCodimOne} using Corollary \ref{maincor:highdeg}.

\subsection{Acknowledgments} I would like to thank Andrew Putman for suggesting I look into the cohomology of $S$-arithmetic groups. I would like to thank him, as well as Tatiana Abdelnaim, Benson Farb, Nir Gadish, Alexander Kupers, Peter Patzt, Yi Shan, and Kevin Watmough, for helpful conversations and encouragement. I would especially like to thank Mathilde Gerbelli-Gauthier for productive discussions about the theory of automorphic representations. I would also like to thank her and Alexander Kupers for feedback on an earlier version of this article.

\section{Topology of $B\SL_n(\bbZ[\tfrac{1}{p}])$}

\subsection{Review of Buildings and $BN$-pairs.}

We begin by recalling the definitions of buildings and Tits systems, following \cite{BrownBuildings}. These definitions will provide a useful language for later sections.

\begin{defn}
    Let $G$ be a group. The data of a \emph{Tits system} is a pair of subgroups $B,N$ and a set $\Sigma$ such that:
\begin{enumerate}
    \item $B,N$ together generate $G$.
    \item $N$ normalizes $T= B\cap N$.
    \item $\Sigma$ is a set of generators of $W=N/T$ such that:
    \begin{enumerate}
        \item For every $s\in \Sigma$ and $w\in W$, we have \[BsBwB\subseteq BwB\cup BswB.\]
        \item For every $s\in \Sigma$, $sBs^{-1}\not\subseteq B.$
    \end{enumerate}
\end{enumerate}
\end{defn}

If the group $G$ has a given Tits system, then a subgroup $P\leq G$ is called \emph{special} if it is of the form $BW'B,$ where $W'$ is a subgroup of $W$ generated by a subset of $\Sigma$. These special subgroups are self-normalizing. 

The poset of left cosets of special subgroups, ordered by reverse inclusion, is the poset of simplices of a simplicial complex $\Delta$, the \emph{building associated to the Tits system.} The homotopy type of this space is known:

\begin{thm}[Solomon-Tits Theorem \cite{Solomon}]
    Let $\Delta$ be the building associated to a Tits system on $G$. If $W$ is infinite, then $\Delta$ is contractible. If $W$ is finite, then $\Delta$ homotopy equivalent to a wedge of spheres of dimension $\abs{\Sigma}-1$.
\end{thm}

We will need two examples, which we state here. For verification of the axioms, see \cite{BrownBuildings}.
\begin{exmp}\label{exmp:SphericalBN}
    Let $\bbF$ be a field, and let $G=\SL_n(\bbF).$ Let $T=T_n(\bbF)$ be the subgroup of diagonal matrices of determinant 1. Let $B = B_n(\bbF)$ be the subgroup of upper-triangular matrices (where the diagonal entries need not be 1, so that $T\leq B$), and let $N=N(T)$ be the set of monomial matrices. Then, $W = N/T\cong S_n$, and we choose $\Sigma = \{s_1,\dots, s_{n-1}\}$ where $s_i$ is the adjacent transposition $(i,i+1).$ Then, $(G,B,N,\Sigma)$ is a Tits system.
\end{exmp}
The special subgroups associated to this Tits system are exactly the standard parabolic subgroups.

Let $\cT_n(\bbF)$ be the building associated to this Tits system. By the Solomon-Tits theorem, this is homotopy equivalent to a wedge of spheres, of dimension $n-2$. The special subgroups of $G$ are stabilizers of flags, so we can also describe $\cT_n(\bbF)$ as the geometric realization of the poset of linear subspaces of $\bbF^n.$ We define:
\begin{defn}\label{defn:Steinberg}
    If $\bbF$ is a field, define the \emph{Steinberg representation} $\St_n(\bbF)$ to be $\tilde{\HH}_{n-2}(\cT_n(\bbF);\bbC).$ This is a representation of $\SL_n(\bbF)$ in a natural way.
\end{defn}
\begin{rem}
    It more standard to define the Steinberg module as $\tilde{\HH}_{n-2}(\cT(\bbF);\bbZ).$ Since it is more convenient for us to use homology and cohomology with $\bbC$ coefficients, we choose to systematically suppress ``$\tensor\bbC$" from notation.
\end{rem}

\begin{rem}
    When $R$ is a commutative ring, the author constructs a \emph{Tits complex} $\cT_n(R)$ extending the definition above \cite{S}. This is functorial in $R$, and it will be occasionally convenient to use this functor to construct maps. This Tits complex is rarely a building in the sense above, unless $R$ is a field, but for number rings and finite rings it satisfies the same connectivity guaranteed by the Solomon-Tits theorem \cite[Theorem A]{S}.
\end{rem}

When $L$ is a $p$-adic field, we can define another Tits system on $\SL_n(L)$.

\begin{exmp}[Iwahori--Matsumoto \cite{IwahoriMatsumoto}]
    Let $L$ be a finite extension of $\bbQ_p$, and let $A$ be its ring of integers. Let $\pi$ be a generator of the maximal ideal of $A$, and let $A/\pi A \cong \bbF_q$ be the residue field. In Example \ref{exmp:SphericalBN}, we found a Tits system $(G,B,N,\Sigma)$ on $\SL_n(L)$ . We will find a second Tits system $(G,\tilde{B},N,\tilde{\Sigma}),$ with data specified as follows. We once again take $N = N_G(T_n(L))$. Let $\tilde{B}$ be the preimage of $B_n(\bbF_q)$ under the natural map $\SL_n(A)\to \SL_n(\bbF_q).$ Then, we have $\tilde{B}\cap N = T_n(A)$. We have an extension
    \[1 \to T_n(L)/T_n(A) \to N/T_n(A) \to N/T_n(L) \to 1.\] We write $\tilde{W} = N/T_n(A)$ and $W = N/T_n(L),$ and observe that $T_n(L)/T_n(A) \cong \bbZ^{n-1}.$ In fact, the above sequence splits, so that $\tilde{W} \cong \bbZ^{n-1} \rtimes W.$ The action of $W$ on $\bbZ^{n-1}$ is its defining action as a Coxeter group. We take $\tilde{\Sigma} = \Sigma \cup \{s_\pi\},$ where an element $s\in \Sigma$ is identified with $(0,s)\in \bbZ^{n-1} \rtimes W.$ The new element $s_\pi$ is the image of $\begin{bmatrix} 0 & \pi \\ -\pi^{-1} & 0 \end{bmatrix}$ under the map $\SL_2(L)\to \SL_n(L)$ corresponding to the highest root.
\end{exmp}
A special subgroup corresponding to this Tits system is called a \emph{standard parahoric subgroup.}

Call the building associated to this Tits system $X(L).$ Since $W$ is infinite, $X(L)$ is contractible.

\subsection{Orbit Spaces of $X(\hat{F}_\fp)$}

Fix a number field $F.$ Let $S_f$ be a nonempty set of finite places, set $S= S_f\cup \Omega_\infty,$ and let $\cO_{F,S}$ be the ring of $S$-integers. Define \[\displaystyle X = \prod_{\fp\in S_f} X(\hat{F}_\fp).\]

We are interested in two orbit spaces of $X$: the usual orbit space $X/\SL_n(\cO_{F,S}),$ and the homotopy orbit space $X_{h\SL_n(\cO_{F,S})}.$ The homotopy type of the second is easily computed:
\begin{prop}
    $X_{h\SL_n(\cO_{F,S})}$ is a $K(\SL_n(\cO_{F,S}), 1)$ space.
\end{prop}
\begin{proof}
    By definition, \[X_{h\SL_n(\cO_{F,S})} = (X\times E\SL_n(\cO_{F,S}))/\SL_n(\cO_{F,S}).\] Since $X$ is contractible, $(X\times E\SL_n(\cO_{F,S}))$ is a free contractible $\SL_n(\cO_{F,S})$-space.
\end{proof}

To study the usual orbit space, we use the following classical fact:

\begin{lem}\label{lem:density}
    $\SL_n(\cO_{F,S})$ is dense in $\prod_{\fp\in S_f} \SL_n(\hat{F}_\fp)$. 
\end{lem}
\begin{proof}
    The group $\SL_n$ has the absolute Strong Approximation property by \cite[Theorem 7.12]{PlatRap}. It follows that $\SL_n(\cO_{F,S})$ is dense in $\prod_{\fp\in S_f} \SL_n(\hat{F}_\fp)$ by \cite[Proposition 7.2(2)]{PlatRap}.
\end{proof}

\begin{lem}
    The natural map \[X/\SL_n(\cO_{F,S})\to X/\prod_{\fp\in S_f} \SL_n(\hat{F}_\fp)\cong \prod_{\fp\in S_f}\Delta^{n-1}\] is a homeomorphism.
\end{lem}
\begin{proof}
    Let $P_\fp$ be a parahoric subgroup of $\SL_n(\hat{F}_\fp)$. Then, $P_{\fp}$ is open. It follows that $P = \prod_{\fp\in S_f} P_\fp$ is open, and therefore so are the left cosets of $P$. By Lemma \ref{lem:density}, $\SL_n(\cO_{F,S})$ is dense, so for every left coset $gP,$ there is some $h\in \SL_n(\cO_{F,S})$ such that $gP=hP.$ It follows that there is only one orbit of left cosets of $P$ under the action of $\SL_n(\cO_{F,S}).$ The statement of the lemma follows immediately.
\end{proof}

We restrict to the case $S_\fp=\{\abs{-}_\fp\}\cup\Omega_\infty$. The map $X\times E\SL_n(\cO_{F,S_\fp})\to X$ induces a map \[B\SL_n(\cO_{F,S_\fp})\cong X_{h\SL_n(\cO_{F,S_\fp})}\to X/\SL_n(\cO_{F,S_\fp})\cong \Delta^{n-1}.\]
The preimage of the simplex corresponding to a parahoric $P$ is homotopy equivalent to $B(P\cap \SL_n(\cO_{F,S_\fp})).$ Therefore, we have a Mayer-Vietoris spectral sequence:
\[E_1^{p,q} = \bigoplus_{h(P) = p}\HH^q(P\cap \SL_n(\cO_{F,S_\fp}); \bbQ)\implies \HH^{p+q}(\SL_n(\cO_{F,S_\fp}); \bbQ) \]
The $d_1$ differential of this spectral sequence is induced by the inclusions between these subgroups. A similar spectral sequence is constructed and studied in the equicharacteristic case in \cite{Knudson}.

We show that the groups appearing on the $E_1$-page are arithmetic.
\begin{defn}
    Let $H$ be a subgroup of $\SL_n(\cO/\fp).$ Then, define $\Gamma_n^H(\fp)$ to be the preimage of $H$ in $\SL_n(\cO)$ under the reduction mod $\fp$ map.
\end{defn}
\begin{lem}\label{lem:arithmeticStabs}
  Let $F$ be a number field. If $P$ is a standard parahoric subgroup of $\SL_n(\hat{F}_p)$ of height $k,$ then there is some parabolic subgroup $Q\subseteq\SL_n(\cO/\fp)$ of height $k-1$ such that $P\cap\SL_n(\cO_{F,S_\fp})\cong \Gamma_n^Q(\fp).$   
\end{lem}
\begin{proof}[Proof of Lemma \ref{lem:arithmeticStabs}]
    For $i\in \{1,\dots, n-1,\pi\},$ let \[M_i = \langle \tilde{B}, s_1,\dots, \hat{s_i}, \dots, s_\pi\rangle\] denote the stabilizer of the $i$th vertex. For every standard parahoric subgroup $P$, there is some $i$ such that $\tilde{B}\subseteq P \subseteq M_i$. We will show that $M_i\cap \SL_n(\cO_{F,S_\fp})$ is isomorphic to $\SL_n(\cO_F)$, and that this isomorphism takes $\tilde{B}\cap \SL_n(\cO_{F,S_\fp})$ to $\Gamma_n^{B_n(\cO/\fp)}(\fp)$. This will imply that $P \cap \SL_n(\cO_{F,S_\fp})$ is a congruence subgroup of the required form.

    First, we consider $M_\pi$. In this case, the identity map will supply the needed isomorphism. Observe that $\tilde{B}\cap\SL_n(\cO_{F,S_\fp})$ lies in $\SL_n(\hat{\cO}_{\fp})\cap \SL_n(\cO_{F,S_\fp}) = \SL_n(\cO_F),$ and in fact is $\Gamma_n^{B_n(\cO/\fp)}(\fp).$ This group contains $B_n(\cO_F)$ and $\Gamma_n(\fp)$. Since $M_\pi\cap\SL_n(\cO_{F,S_\fp})$ contains $\Gamma_n(\fp)$, it corresponds to a subgroup of $\SL_n(\cO/\fp).$ Since $M_\pi\cap\SL_n(\cO_{F,S_\fp})$ contains $B_n(\cO_F)$ and a generating set for the Weyl group, its corresponding group in $\SL_n(\cO/\fp)$ contains the elementary matrices. Since these matrices generate $\SL_n(\cO/\fp)$, we have that $M_\pi\cap\SL_n(\cO_{F,S_\fp}) = \SL_n(\cO_{F})$. 

    For $i\in \{1,\dots, n-1\},$ we will obtain our desired isomorphism by conjugating within $\GL_n(\cO_{F,S_\fp}).$ Such a conjugation map will preserve $\SL_n(\cO_{F,S_\fp}).$ Let $r$ be the permutation matrix of $(1,2,\dots, n)$, and let $t_i = \text{diag}(\pi,\dots, \pi, \underbrace{1,\dots, 1}_{i\text{ times}}).$ Then, $M_i^{r^it_i}\cap \SL_n(\cO_{F,S_\fp})= \SL_n(\cO_F) $, since this conjugation takes $\{s_1,\dots, \hat{s_i}, \dots, s_\pi\}$ to $\{s_1,\dots, s_{n-1}\}$ and stabilizes $\tilde{B}$. It follows that $M_i\cap \SL_n(\cO_{F,S_\fp})$ is isomorphic to $\SL_n(\cO_F),$ via an isomorphism carrying $\tilde{B}\cap \SL_n(\cO_{F,S_\fp})$ to $\Gamma_n^{B_n(\cO/\fp)}(\fp).$
\end{proof}

We obtain the following:
\begin{thm}\label{thm:SpectralSequence}
    There is an $E_1$ spectral sequence
    \[E_1^{p,q} = \bigoplus_{h(P) = p}\HH^q(P\cap \SL_n(\cO_{F,S_\fp}); \bbC)\implies \HH^{p+q}(\SL_n(\cO_{F,S_\fp}); \bbC). \]
    We have $E_1^{n-1,q} = \HH^q(\Gamma_n^{B_n(\cO/\fp)}(\fp);\bbC),$ and $E_1^{0,q} = \HH^q(\SL_n(\cO/\fp);\bbC)^n$. If $0<p<n-1,$ we have that each summand of $E_1^{p,q}$ is of the form $\HH^q(\Gamma_n^{Q}(\fp);\bbC),$ for some parabolic $Q\subseteq \SL_n(\cO/\fp)$ of height $p+1$.
\end{thm}
In the decomposition of $E_1^{p,q}$, the same parabolic may occur multiple times. These arithmetic groups are of number-theoretic interest, and the their cohomology is not well-understood in high degrees. See \cite{Abdelnaim} for some interesting results in this direction. 

Finally, we observe that $E_1^{p,q}$ vanishes when either $p$ or $q$ is large:
\begin{cor}\label{cor:specSeqBound}
    In the spectral sequence of Theorem \ref{thm:SpectralSequence}, $E_1^{p,q} = 0$ if $p>n-1$ or $q>\cd_{\bbQ} \SL_n(\cO).$
\end{cor}
The vanishing when $p$ is large follows from the structure of the Mayer--Vietoris spectral sequence. The vanishing when $q$ is large follows from Borel--Serre's computation of the cohomological dimension of arithmetic groups \cite{BS}.
\section{Cohomology of $S$-arithmetic Groups}

In this section, we show how to use the automorphic decomposition of cohomology to describe the cohomology of $S$-arithmetic groups. The organization is as follows. In section 3.1, we collect some results about differentiable cohomology and relative Lie algebra cohomology. We then provide an expository overview of Matsushima's formula, which provides the automorphic decomposition in the cocompact case. This is not the case of interest for the paper, but this overview serves as motivation for later sections, especially for the reader who is not already familiar with this area. In \S 3.2, we discuss the particular features of the noncompact $S$-arithmetic case. In particular, we use the theorem of Blasius--Franke--Grunewald that the cohomology in this case agrees with the so-called $L^2$-cohomology. Combining this with results of Borel--Labesse--Schwermer and Wallach, we show how to obtain a decomposition result similar to Matsushima's formula. In \S 3.3, we mildly generalize a previously known result, showing that cuspidal cohomology only contributes in a range of dimensions. This will allow us to completely describe the cohomology of a finite-index subgroup of $\SL_n(\cO_{F,S})$ outside of this range.

\subsection{The Automorphic Decomposition}
Let $F$ be a number field, let $S_f$ be a finite set of finite places of $F$, and set $S = S_f\cup\Omega_\infty$. Let $\cG$ be an almost simple and simply-connected algebraic group scheme defined over $F$. See \cite[II.1.6]{Jantzen} for a definition. For example, we may take \[\cG\in \{\SL_n, \Sp_{2g}, \SO_{p,q}, \SU_{p,q}\}.\] If $\nu \in S$, let $G_\nu = \cG(\hat{F}_\nu)$, and define $G = \prod_{\nu\in S} G_\nu$. We also define $G_f = \prod_{\nu\in S_f} G_\nu$ and $G_\infty = \prod_{\nu\in \Omega_\infty}\cG(\hat{F}_\nu) = \cG(F\tensor\bbR).$ In particular, $G = G_\infty\times G_f$. We will often view $G_\infty$ as a real Lie group. Let $\fg$ be the Lie algebra of $G_\infty$, and let $K$ be a maximal compact subgroup. 

Let $\Gamma$ be an irreducible lattice in $G$. Let $E$ be a rational representation of $G_\infty$, which we regard as a representation of $G$ by extending trivially, then as a representation of $\Gamma$ by restriction. We are interested in describing the cohomology groups $\HH^q(\Gamma;E).$

We will translate this into the language of differentiable cohomology. We define this here, and refer to \cite{BW} for a systematic development.

\begin{defn}
    Let $M_\pi = (M,\pi)$ be a topological vector space $M$ over $\bbC$ with a continuous linear action $\pi$ of $G$. We say that a vector $v\in M$ is \emph{smooth} if, for every $v'\in G\cdot v,$ we have that the orbit map $G_\infty\to G_\infty\cdot v'$ is a smooth map of manifolds and that the stabilizer of $v'$ in $G_f$ is open. We write $M_\pi^\infty$ for the subspace of smooth vectors of $M_\pi$, and we say that $M_\pi$ is \emph{smooth} if $M_\pi=M_\pi^\infty$.
\end{defn} 
Let $\cC_G^\infty$ be the category of smooth representations of $G$. By \cite[Proposition XII 1.5]{BW}, we may form the \emph{differentiable cohomology} as the derived functor \[\HH^q_d(G,-) = R^q\Hom_{\cC^\infty_G}(\bbC,-).\]
\subsubsection{Relative Lie algebra cohomology}
The case where $G_f$ is trivial and $G = G_\infty$ will be important for us, so we pause to consider it. 
\begin{defn}
    Let $(M,\pi)$ be a smooth $K$-module, also admitting an action $\rho$ of $\fg.$ We say that $(M,\pi,\rho)$ is a \emph{$(\fg,K)$-module} if for all $k\in K,$ $X\in U(\fg),$ $v\in M$, it satisfies
    \begin{align*}
        \pi(k)\cdot \rho(X)\cdot v = \rho(\ad(k)X)\cdot \pi(k)\cdot v
    \end{align*}
    and \[d\pi = \rho|_\fk.\]
\end{defn}
Every smooth $G_\infty$-module defines a $(\fg,K)$-module by restriction. For a general $G_\infty$-module, we define its \emph{Harish-Chandra module}: 
\[M_{0} = \Span\{v\in M^\infty\mid \dim \Span (K\cdot v)<\infty\}.\]
This is not a $G_\infty$-submodule of $M$, but it is a $(\fg,K)$-module.
\begin{defn}
Given a $(\fg,K)$-module $M,$ define the \emph{$(\fg,K)$-cohomology:}
\[\HH^*(\fg,K;M) = \HH^* (\Hom_K(\Lambda^q(\fg/\fk),M)),\]
where the differential is \begin{align*}
    d\omega(X_0,\dots, X_p)  = &\sum_{i=0}^p (-1)^i X_i\cdot \omega(X_0,\dots, \widehat{X_i},\dots, X_p)  \\ &+\sum_{0\leq i<j\leq p} (-1)^{i+j} \omega([X_i,X_j], X_0,\dots, \widehat{X_i}, \dots, \widehat{X_j},\dots, X_p).
\end{align*}
Given a $G_\infty$-module $N$, we define the $(\fg,K)$-cohomology using the Harish-Chandra module:
\[\HH^*(\fg,K;N) = \HH^*(\fg,K;N_0).\]
\end{defn}
If $M$ is a smooth $G_\infty$-module, then there is a natural isomorphism: \[\HH^*(\fg,K; M|_{(\fg,K)}) \cong \HH^*(\fg,K;M_0).\]
Therefore, there is no ambiguity when speaking of the $(\fg,K)$-cohomology of $M$.

We have the following result:
\begin{prop}[{\cite[IX.5.6 (ii)]{BW}}]
    Let $G=G_\infty$, and let $M$ be a smooth Fr\'echet module. Then, there is an isomorphism with the relative Lie algebra cohomology of $M$:
\[ \HH^q_d(G_\infty;M) \cong \HH^q(\fg, K; M).\]
\end{prop}
 
If $M$ is a $(\fg,K)$-module, we say it has a \emph{central character} if the restriction of $M$ to $Z(K)\cap \ker (\ad)$ has a single isotypic component corresponding to a character $\omega_M$. It has an \emph{infinitesimal character} if there is a functional $\chi_M:Z(\fg)\to \bbC$ such that, for all $z\in Z(\fg),$ we have \[z\cdot v = \chi_M(z) v.\] A $(\fg,K)$-module $M$ is called \emph{admissible} if, as a representation of $K$, each isotypic component is finite dimensional. As a consequence of the Peter-Weyl theorem, if $M$ is the Harish-Chandra module of an irreducible unitary representation of $G_\infty,$ then $M$ is admissible. If $M$ is both irreducible and admissible, then it has a central character and an infinitesimal character \cite[\S 0.2.7]{BW}. 
\begin{thm}[{\cite[I.5.3]{BW}}]
    Let $U,V$ be $(\fg,K)$-modules, and assume that each has a central character and an infinitesimal character. Assume that either $V$ or $U$ is finite-dimensional. Then, \[\HH^*(\fg,K;U\tensor V) = 0\] unless $\omega_U = \omega_{V^*}$ and $\chi_U = \chi_{V^*}.$
\end{thm}

If $M$ is a unitary representation of $G_\infty$, we say that $M$ is \emph{cohomological} if, for some rational representation $E$ of $G_\infty$, we have $\HH^*(\fg, K; M\tensor E)\ne 0$. There are only finitely many cohomological representations for any given $E$. These representations are classified, and their $(\fg,K)$-cohomology groups are computed, by Vogan--Zuckerman \cite{VoganZuckerman}. 

When $M=\bbC,$ the cohomology was computed by, for example, Borel \cite{BorelStable}. It agrees with the cohomology of $L/K,$ where $L$ is a compact real form of $G$ containing K. For our primary application, we have the following calculation:

\begin{prop}[{\cite[\S 10]{BorelStable}}]\label{prop:gKCoh} Let $G_\infty =\SL_n(F\tensor\bbR) = \SL_n(\bbR)^r\times \SL_n(\bbC)^s,$ and set $K = \SO_n^r\times \SU_n^s.$ Then,
    \[\HH^q(\fg,K;\bbC) \cong \HH^q((\SU_n/\SO_n)^r\times (\SU_n)^s; \bbC).\]
\end{prop}
The cohomology of the latter compact manifold -- which we denote $M_{r,s}$ -- is well understood in the range we are interested in. We have \[\dim M_{r,s} = r\cdot \binom{n+1}{2} + s\cdot (n^2-1) = \nu_F(n) + (n-1),\] where $F$ is any field of signature $(r,s).$ We record the cohomology in high degrees.
\begin{cor}\label{cor:gKCoh} Let $n\geq 2$, and let $G_\infty, K$ be as in Proposition \ref{prop:gKCoh}. Let $F$ be a field of signature $(r,s)$. Then,
    \[\HH^{\nu_F(n)+(n-1)-i}(\fg,K;\bbC)=\begin{cases}
        \bbC & i=0 \\ 0 & i = 1 \\ \bbC^{s} & i=2
    \end{cases}\]
\end{cor}

\subsubsection{The Cocompact Case}

Let $C(\Gamma\backslash G)$ be the space of continuous $\bbC$-valued functions on $\Gamma\backslash G$, given the topology of uniform covergence. This is a Fr\'echet space. Define $C^\infty(\Gamma\backslash G) = (C(\Gamma\backslash G))^\infty$. This is simply the module of smooth functions when $S=\Omega_\infty.$ By a variant of Shapiro's lemma \cite[XIV 1.2]{BW}, there is an isomorphism \[\HH^q(\Gamma;E)\cong \HH^q_d(G; C^\infty(\Gamma\backslash G)\tensor E).\] 

To proceed, we want to analyze the structure of $C^\infty(\Gamma\backslash G)$. We begin by sketching the case where $\Gamma$ is a cocompact lattice. In the primary application of interest for this paper, $\Gamma$ is not cocompact, but the cocompact case helps motivate the strategy for the general case.

Since $\Gamma\backslash G$ is compact, we have that $C^\infty(\Gamma\backslash G) = L^2(\Gamma\backslash G)^\infty$. This latter space is decomposed using techniques from harmonic analysis by a theorem of Gelfand--Piatetski-Shapiro. See for example \cite[Theorem 9.2.2]{DeitmarEchterhoff} for a proof.
\begin{thm}[Gelfand--Piatetski-Shapiro]
    Let $G$ be a locally compact unimodular group, and let $\Gamma$ be a cocompact lattice in $G$. Writing $\hat{G}$ for the unitary dual of $G$, we have a Hilbert space direct sum: \[L^2(\Gamma\backslash G) = \widehat{\bigoplus_{\pi\in \hat{G}}} (V_\pi)^{ m_{\Gamma,\pi}}\] where $ m_{\Gamma,\pi}$ is finite for all $\Gamma,\pi$.
\end{thm}

Applying this result to cohomology, one extracts the following formula:

\begin{thm}[{\cite[XIII.1.5]{BW}}]
    Let $G=G_\infty\times G_f$ be as above, and let $\Gamma$ be a cocompact lattice in $G$. We have an algebraic direct sum:
    \[\HH^q(\Gamma;E) = \bigoplus_{\pi\in \hat{G}}\HH^q_d(G; V^\infty_\pi\tensor E)^{m_{\Gamma,\pi}}. \] All but finitely many terms of this direct sum vanish.
\end{thm}
The fact that all but finitely many terms of the direct sum vanish can be extracted from the finite-dimensionality of $\HH^q(\Gamma;E)$ \cites{BS,BoSe2}. It can also be extracted from Vogan--Zuckerman theory \cite{VoganZuckerman} and Proposition \ref{prop:Casselman} below.

We conclude by further characterizing the cohomology groups on the right-hand side. We have that $V_\pi = V_{\pi_\infty}\tensor V_{\pi_f},$ and these representations are unitary and hence admissible. Then, by \cite[XII.2.6 and XII.3.1]{BW}, we have:

\[\HH^*_d(G; V^\infty_\pi\tensor E) = \HH^*_d(G_\infty; V^\infty_{\pi_\infty}\tensor E)\tensor \HH^*_d(G_f; V^\infty_{\pi_f})\]

The apparent asymmetry in the coefficients on the right-hand side is due to the fact that $E$ is pulled back from $G_\infty$. Since $V_\pi$ is unitary and lies in $L^2$, we have that $V_\pi^\infty\tensor E$ is smooth and Fr\'echet, so $\HH^*_d(G_\infty; V^\infty_{\pi_\infty}\tensor E)$ equals the relative Lie algebra cohomology $\HH^*(\fg,K; V^\infty_{\pi_\infty}\tensor E)$. Casselman observed that the differentiable cohomology of $G_f$ has a quite simple form:
\begin{prop}[{\cite[XI.3.9]{BW}}]\label{prop:Casselman}
    Let $\nu\in S_f$, and let $V$ be an irreducible admissible unitary representation of $G_\nu$. Then, $\HH^*_d(G_\nu; V^\infty) = 0$ unless $V$ is trivial or the special representation\footnote{Often these are called \emph{Steinberg representations}, but they are not the same Steinberg representations which occur elsewhere in this paper. We hope this choice of nomenclature is less confusing.} $\Sp(G)$ of $G$. We have:
    \begin{align*}
        \HH^q_d(G_\nu; \bbC) &= \begin{cases}
            \bbC & q=0 \\ 0 & \text{otherwise.}
        \end{cases}\\
        \HH^q_d(G_\nu; \Sp(G)^\infty) &= \begin{cases}
            \bbC & q=\rk(G) \\ 0 & \text{otherwise.}
        \end{cases}
    \end{align*}
\end{prop}

Finally, we need the following consequence of Strong Approximation.
\begin{prop}[{\cite[XIII.4.3]{BW}}]
    Assume that $G_\nu$ is not compact for any Archimedian place $\nu\in \Omega_\infty$. Let $V$ be a unitary irreducible representation of $G$ appearing in $L^2(\Gamma\backslash G)$. Write $V = \bigotimes_{\nu\in S} V_\nu.$ If $V_\nu$ is trivial for a place $\nu\in S$ where $G_\nu$ is not compact, then $V$ is trivial.
\end{prop}
When $\cG$ is not simply connected, this result holds under the additional hypothesis that $V$ is cohomological.

Combining the results of this section, we have:
\begin{thm}[{\cite[XIII.4.4]{BW}}]
Let $\cG$ be an algebraic group which is absolutely almost simple over $\hat{F}_\nu$, for all $\nu\in S$. Let $E$ be a rational representation of $G_\infty,$ and let $\Gamma$ be a cocompact lattice in $G$. Then,
    \[\HH^q(\Gamma;E) = \HH^q(\fg,K;E)\oplus\bigoplus_{\mathbbm{1}\ne\pi\in \hat{G}}\HH^{q-\sum_{\nu\in S_f}\rk(G_\nu)}(\fg,K; V^\infty_\pi\tensor E)^{m_{\Gamma,\pi}}. \] All but finitely many terms of this direct sum vanish. If $\abs{S_f}\ne 0,$ then $V_\pi$ has nonzero contribution only if $V_{\pi_\nu}$ is the special representation of $G_\nu$ for all $\nu\in S_f.$
\end{thm}
In the arithmetic case, this is known as \emph{Matsushima's formula.}
\subsection{The non-cocompact $S$-arithmetic case.} We would like a similar decomposition to Matsushima's formula in the non-cocompact case. The first difficulty is that $L^2(\Gamma\backslash G)$ does not decompose as a Hilbert space direct sum of irreducible representations. Define the \emph{discrete spectrum} $L^2(\Gamma\backslash G)_{disc}$ be the closure of the subspace of $L^2(\Gamma\backslash G)$ spanned by irreducible subrepresentations of $G$, and let the \emph{continuous spectrum} $L^2(\Gamma\backslash G)_{cts}$ be its orthogonal complement. By theorems of Borel--Casselman \cite[Theorem 4.5]{BorelCasselman} and Borel--Labesse--Schwermer \cite[\S 7]{BLS}, the continuous spectrum does not contribute to the cohomology of an $S$-arithmetic lattice $\Gamma$. Therefore, $\HH^*_d(G;L^2(\Gamma\backslash G)\tensor E)$ admits a similar direct sum decomposition to that of Matsushima's formula.

The second difficulty lies in relating $C^\infty(\Gamma\backslash G)$ with $L^2(\Gamma\backslash G).$ This is done by a theorem of Franke \cite[Theorem 18]{F}, resolving a conjecture of Borel. 

Define $\cA_E(\Gamma\backslash G)$ to be the space of functions $f\in C^\infty(\Gamma\backslash G)$ which are $K$-finite, are killed by a power of $\ker(\chi_{E^*}),$ and have uniform moderate growth in the sense of \cite[\S 2.1]{BLS}. We have  \[ \HH^*(\Gamma; E) \cong \HH^*_d(G; \cA_E(\Gamma\backslash G)\tensor E).\]

Let $L^2_E(\Gamma\backslash G) = \cA_E(\Gamma\backslash G)\cap L^2(\Gamma\backslash G)^\infty$, and define \[(\cA/L^2)_E(\Gamma\backslash G) = \cA_E(\Gamma\backslash G)/L^2_E(\Gamma\backslash G).\] In \cite[\S 6]{F}, Franke finds a filtration of $(\cA/L^2)_E(\Gamma\backslash G)$, whose quotients are characterized in terms of representations induced from parabolic subgroups. 
In the arithmetic case, both $L^2_E(\Gamma\backslash G)$ and $(\cA/L^2)_E(\Gamma\backslash G)$ contribute to the cohomology of $\Gamma.$ Interesting examples include the main theorems of \cite{CalegariBoxerGee} or \cite{Schwermer}. We additionally point out an interesting interaction.

\begin{exmp}
    Let $\Gamma = \SL_n(\bbZ)$, and $G=\SL_n(\bbR).$ Then, $\Gamma\backslash G$ has finite volume, so the constant functions lie in $L^2(\Gamma\backslash G),$ and additionally lie in $L^2_\bbC(\Gamma\backslash G).$ It is not hard to see that every function lying in a trivial representation must be constant, so we have $m_{\Gamma,\mathbbm{1}} = 1.$ We have \[\HH^*(\fg,K;\bbC) \cong \HH^*(\SU_n/\SO_n; \bbC).\] The compact manifold on the right-hand-side has dimension $\binom{n}{2} + n-1$, and therefore \[\HH_d^{\binom{n}{2}+n-1}(G;L^2_{\bbC}(\Gamma\backslash G)^\infty)\ne 0.\] But by Borel--Serre's theorem \cite{BS}, we have that if $q\geq \binom{n}{2},$ then \[\HH^q(\Gamma; \bbC) \cong \HH^q_d(G;\cA_\bbC(\Gamma\backslash G)) = 0.\] Therefore, in this range we must have that the connecting homomorphism induces an isomorphism:
    \[\HH_d^{q+1}(G; (\cA/L^2)_{\bbC}(\Gamma\backslash G)) \cong \HH_d^{q}(G; L^2_{\bbC}(\Gamma\backslash G)^\infty).\]
    In general, both groups will be nonzero.
\end{exmp}
\begin{rem}
    The Church--Farb--Putman vanishing conjecture \cite{CFPConjecture} states that if $i\leq n-2,$ then \[\HH^{\binom{n}{2} - i}(\SL_n(\bbZ);\bbC) = 0.\] This holds if and only if the connecting homomorphism \[\HH_d^{q+1}(G; (\cA/L^2)_{\bbC}(\Gamma\backslash G)) \to \HH_d^{q}(G; L^2_{\bbC}(\Gamma\backslash G)^\infty)\] is an isomorphism for $q\geq \binom{n}{2}-n+2.$
\end{rem}

We have seen that $(\cA/L^2)_E(\Gamma\backslash G)$ plays a profound role in the cohomology of an arithmetic group. By a theorem of Blasius--Franke--Grunewald, it plays no role in the cohomology of a properly $S$-arithmetic group:

\begin{thm}[{\cite[Proposition 1]{BFG}}]
    Let $S\ne \Omega_\infty$. Then, \[\HH_d^*(G; C^\infty(\Gamma\backslash G)\tensor E) \cong \HH_d^*(G; L^2_E(\Gamma\backslash G)\tensor E).\]
\end{thm}

We state the induced decomposition of $\HH^q(\Gamma; E)$.
\begin{thm}[{\cite[Theorem 1]{BFG}}]\label{thm:S-arithMatsushima}
    Let $F$ be a number field, and let $S\ne\Omega_\infty$ be a finite set of places. Let $\cG$ be a semi-simple and simply-connected group scheme. Then,
    \[\HH^q(\Gamma;E) \cong \HH^q(\fg,K; E) \oplus \bigoplus_{\pi\in \hat{G}} \HH_d^{q-\sum_{\nu\in S_f}\rk(G_\nu)}(G_\infty; V_\pi\tensor E)^{m_{\pi,\Gamma}}.\]
    Only finitely many terms of this sum are nonzero. In particular, if the term corresponding to $\pi$ is nonzero, then if we write $\pi = \pi_\infty\tensor\bigotimes_{\nu\in S_f} \pi_\nu,$ we must have that $\pi_\nu$ is the special representation of $\cG(\hat{F}_\nu)$ for all $\nu\in S_f$.
\end{thm}
See also the main theorem of Borel--Labesse--Schwermer \cite{BLS}.

There is one additional consequence of this theorem. Let $P$ be a proper parabolic subgroup of $G$ defined over $F$, and let $U_P\to P \to L_P$ be its Levi decomposition. We write: \[\Gamma_P = \Gamma\cap P \qquad \qquad \Gamma_{U_P} = \Gamma\cap U_P \qquad \qquad \Gamma_{L_P} = \Gamma_P/\Gamma_{U_P}.\]

\begin{defn}
    A function $f\in L^2(\Gamma\backslash G)$ is \emph{cuspidal} if, for every proper parabolic subgroup $P$ defined over $F$, the function $f^P:\Gamma_{L_P}\backslash L_P\to \bbC$ defined by the Haar integral \[ f^P(x)= \int_{\Gamma_{U_P}\backslash U_P} f(n\cdot x) \ dn\] vanishes everywhere. Write $L^2(\Gamma\backslash G)_{cusp}$ for the space of cuspidal functions in $L^2(\Gamma\backslash G)$.
\end{defn}

We have the following theorem: 
\begin{thm}[{\cite[\S 4]{BFG} or \cite[\S 6.4]{BLS}}]\label{thm:Wallach}
    Let $V_\pi$ be a nontrivial irreducible unitary representation of $G$ such that the term of Theorem \ref{thm:S-arithMatsushima} corresponding to $\pi$ is nonzero. Then, the full multiplicity of $V_\pi$ lies in $L^2(\Gamma\backslash G)_{cusp}.$
\end{thm}
This follows from a generalization of a theorem of Wallach \cite[Theorem 4.3]{Wallach} to finite places, along with the fact that the special representation of $\cG(\hat{F}_\nu)$ is tempered \cite[XI.2.6 and 2.14]{BW}.

\subsection{Proofs of Theorem \ref{mainthm:automorphic} and Corollary \ref{maincor:highdeg}} 

We specialize the above theory to the case $\cG = \SL_n$. In this case, if $F$ has signature $(r,s)$, we have $G_\infty = \SL_n(\bbR)^r\times \SL_n(\bbC)^s$. Therefore, the first summand of Theorem \ref{thm:S-arithMatsushima} is computed by Proposition \ref{prop:gKCoh}.

It remains to understand the contribution of cuspidal cohomology. We must first recall some terminology for unitary representations of Lie groups.

If $V_\pi$ is a unitary representation of a reductive Lie group $G_\infty$, a \emph{matrix coefficient} is a function \[g\mapsto \langle \pi(g)\cdot v,w\rangle \] for some choice of $v,w\in V.$ Let $^0G_\infty$ denote the intersection \[\bigcap_{\chi:G\to \bbR^\times} \ker(\abs{\chi}).\] An irreducible unitary representation is \emph{discrete series} if the restriction of every matrix coefficient lies in $L^2(^0G_\infty).$ If $G_\infty$ is semisimple, a theorem of Harish-Chandra \cite{HarishChandraDiscrete} implies that these exist if and only if $\rk {G_\infty} = \rk K.$ Among the reductive groups $\GL_n(\bbR)$ and $\GL_n(\bbC)$, only $\GL_1(\bbR), \GL_1(\bbC),$ and $\GL_2(\bbR)$ have discrete series representations. An irreducible unitary representation $V$ is called \emph{tempered} if every matrix coefficient lies in $L^{2+\epsilon}(^0G_\infty)$ for all $\epsilon>0.$ 

Unitary representations can all be constructed via \emph{normalized parabolic induction}. Let $P$ be a standard parabolic subgroup, and let $P=MAN$ be its Langlands decomposition. Let $\fh$ be a Cartan subalgebra of $\fg$ containing $\fa$, and let $\rho_P$ be the half-sum of the roots of $\fp_\bbC$ with respect to $\fh_\bbC$. Given a unitary representation $\sigma$ of $M$ and any representation $\eta$ of $A$ (identified with a functional on its complexified Lie algebra), define \[I_P^G(\sigma,\eta) = \{f\in C^\infty(G,V_\sigma)\mid f(man\cdot g) = (\rho_P + \eta)(a) \cdot \sigma(m) \cdot f(g) \}.\] The result of this procedure is not necessarily unitary or irreducible, but by work of Langlands, every unitary representation appears as a subquotient of such an induction. However, $I_P^G(\sigma,0)$ is always unitary. If $\sigma$ is discrete series, then $I_P^G(\sigma,0)$ is tempered \cite[IV.3.7]{BW}. It follows that \emph{every} Lie group admits tempered representations.
\begin{rem}
    While this procedure is called induction, the formula is similar to that of coinduction, and indeed it is a right adjoint to the restriction of a smooth representation to a closed subgroup. However, since $G/P$ is always compact for a parabolic $P$, one can prove that it is also a left adjoint to restriction.
\end{rem}

Finally, let $\bbF\in\{\bbR,\bbC\}$, and fix a nontrivial additive character $\psi$ of $\bbF$. Define a character of $U_n(\bbF)$ by \[\psi_U(B) = \psi\left(\sum_{i=1}^{n-1} b_{i,i+1}\right). \] A \emph{Whittaker functional (relative to $\psi$) on $V$} is a continuous map of $U_n(\bbF)$ representations $\lambda:\Res^{G_\infty}_{U_n(\bbF)}V\to \bbC_{\psi_U}.$ A unitary $G_\infty$-representation is called \emph{generic (relative to $\psi$)} if it admits a Whittaker functional relative to $\psi$.

The following result is a mild generalization of a theorem which can be found in Schwermer \cite[Theorem 3.3]{SchwermerHolomorphy}. We give a proof following Clozel \cite[Proof of Lemma 4.9]{Clozel}.

\begin{thm}\label{thm:tempered}
    Let $V_\pi$ be a nontrivial irreducible unitary representation of $G$, appearing in the cuspidal spectrum of $\Gamma\backslash G$. If $\HH_d^*(G;V_{\pi_\infty}\tensor E)\ne 0,$ then $V_{\pi_\infty}$ is a tempered representation of $G_\infty$. We have
    \[\HH^q_d(G_\infty;V_{\pi_\infty}) = 0\]
    if $q\not \in \left[\frac{(\dim (G_\infty/K)) -(\rk G_\infty - \rk K_\infty)}{2},\frac{(\dim (G_\infty/K)) +(\rk G_\infty - \rk K_\infty)}{2}\right]$.
\end{thm}
\begin{proof}
    Let $\hat{F_S} = \prod_{\nu\in S} \hat{F}_\nu$. Observe that $G= \SL_n(\hat{F}_S).$ Our first step is to transfer the problem to $\GL_n(\hat{F}_S)$. There is a diagram of algebraic groups with exact rows:
    \[\xymatrix{\SL_n\ar[r]\ar[d] & \GL_n \ar[r]\ar[d] & \bbG_m \ar[d] \\ \PSL_n \ar[r] & \PGL_n \ar[r] & \bbG_m/n\bbG_m}\]

    We can promote $V_\pi$ to a unitary representation of $\PSL_n(\hat{F}_S).$ This follows by observing that cohomological representations of $G_\infty$ automatically have trivial central character, and that the special representation of a $p$-adic group likewise has trivial central character. 
    
    We now observe that $\PSL_n(\hat{F}_S)$ is a finite index subgroup of $\PGL_n(\hat{F}_S).$ This follows from the calculation of $\bbG_m/n\bbG_m$. At an infinite place this is easy. At a finite place, this follows from the isomorphism $\hat{F}_\nu \cong \bbZ\times \hat{\cO}_{F,\nu}^\times$ and the nonzero radius of convergence of the power series computing $n$th roots. 

    Consider the representation $\Ind_{\PSL}^{\PGL} V_\pi$. This representation may not be irreducible. We will first show that this lies in the cuspidal spectrum. Indeed, the double induction formula identifies $C^\infty(\Gamma\backslash\PGL_n(\hat{F}_S))$ with $\Ind_{\PSL}^{\PGL} C^\infty(\Gamma\backslash\PSL_n(\hat{F}_S)).$ This identification preserves $L^2$-functions, since as a measure space $\PGL_n(\hat{F}_S)$ is the product of $\PSL_n(\hat{F}_S)$ and a finite group. Likewise, this identification preserves cuspidal functions, since the condition on cuspidal functions concerns integrals over subgroups of the group of unipotent upper-triangular matrices, which always lie in $\SL_n$. We therefore have an identification \[\Ind_{\PSL}^{\PGL} L^2(\Gamma\backslash \PSL_n(\hat{F}_S))_{cusp} \cong L^2(\Gamma\backslash \PGL_n(\hat{F}_S))_{cusp}.\] By exactness of induction from finite-index subgroups, it follows that $\Ind_{\PSL}^{\PGL} V_\pi$ lies in $L^2(\Gamma\backslash \PGL_n(\hat{F}_S))_{cusp}.$
    
    The representation $\Ind_{\PSL}^{\PGL} V_\pi$ may not be irreducible, but it is finite-length, and as it lies in the cuspidal spectrum, it is a direct sum of irreducibles. It follows that we can find an irreducible unitary representation $W_\pi$ of $\PGL_n(\hat{F}_S)$ such that $V_\pi$ is isomorphic to a subrepresentation of $\Res_{\PSL}^{\PGL} W_\pi$. Finally, we identify $W_\pi$ with a representation of $\GL_n(\hat{F}_S)$ with trivial central character.

    The remainder of the argument is carried out locally, at each infinite place. We have $ W_{\pi_\infty} = \bigotimes_{\nu\in\Omega_\infty} W_{\pi_\nu},$ where each factor is a irreducible representation of $\GL_n(\hat{F}_\nu)$. Likewise, we have a decomposition $ V_{\pi_\infty} = \bigotimes_{\nu\in \Omega_\infty} V_{\pi_\nu},$ and $V_{\pi_\nu}$ is a summand of $\Res_{\PSL}^{\PGL} W_{\pi_\nu}$. It follows that each $W_{\pi_{\nu}}$ is cohomological \cite[\S 4]{NairPrasad}.

    Fix $\nu_0\in \Omega_\infty$, and let $\bbF = \hat{F}_{\nu_0}\in\{\bbR,\bbC\}$. For every nonzero $ w\in \bigotimes_{\nu\in S\setminus\{\nu_0\}} W_{\pi_\nu},$ there is a map of $G_\infty$-representations $W_{\pi_{\nu_0}}\to W_\pi\to L^2(\Gamma\backslash G)_{cusp}.$ Since $W_{\pi_{\nu_0}}$ is irreducible, this is an injection. Any nonzero function $\xi$ lying in the image of this map generates a $G_\infty$-representation isomorphic to $W_{\pi_{\nu_0}}.$ The Fourier expansion of $\xi$ allows one to construct a Whittaker model for $W_{\pi_{\nu_0}}$, by \cite[Example 14.9.2]{GH}\footnote{Be warned that this example does not specify that the automorphic form needs to be cuspidal, however this condition is needed for the Fourier transform to exist.}.
    
    The unitary dual of $\GL_n(\bbF)$ is known \cites{VoganUnitDual, Tadic}. Further, it is known which elements of the unitary dual of $\GL_n(\bbF)$ are generic. Every generic unitary representation is of the form $I_P^G(\tau, 0),$ where $\tau = \tau_1\tensor\dots \tensor\tau_\ell$ and each $\tau_i$ lies in one of the following four families:
    \begin{itemize}
        \item Unitary characters of $\GL_1(\bbF)$ -- recall that these are discrete series representations.
        \item Discrete series representations of $\GL_2(\bbR).$
        \item Complementary series representations constructed from unitary characters of $\GL_1.$ Given $\chi$ a unitary character of $\GL_1,$ and $0<\alpha<\frac{1}{2},$ we have the irreducible unitary representation $I_{P_{(1,1)}}^{\GL_2} \left(\chi\tensor \chi, \abs{-}^{(\alpha,-\alpha)}\right).$ 
        \item Complementary series representations constructed from discrete series representations of $\GL_2(\bbR).$ If $\sigma$ is such a discrete series representation, and $0<\alpha<\frac{1}{2},$ we have an irreducible unitary representation \[I_{P_{(2,2)}}^{\GL_4}\left(\sigma\tensor \sigma, \abs{-}^{(\alpha,\alpha,-\alpha,-\alpha)}\right).\]
    \end{itemize}

    Complementary series representations do not contribute to cohomological representations, as their infinitesimal characters are not integral. Therefore, $W_{\pi_{\nu_0}} = I_P^G(\tau, 0)$ where $\tau$ is discrete series. It follows that $W_{\pi_{\nu_0}}$ is a tempered representation of $\GL_n(\bbF).$ By the Local Langlands correspondence, $V_{\pi_{\nu_0}}$ is also tempered.

    The desired bounds on the cohomology follow immediately from \cite[III.5.2(iii)]{BW}.
\end{proof}

We finally obtain the following result, which immediately implies Theorem \ref{mainthm:automorphic}.
\begin{thm}
    Let $F$ be a number field, and let $S \ne \Omega_\infty$. Let $\Gamma$ be a finite-index subgroup of $\SL_n(\cO_{F,S}).$ Then, 
    \[\HH^q(\Gamma; \bbC) \cong \HH^q(\fg, K; \bbC)\] 
    for \begin{align*} q\not\in \left[\tfrac{(\dim (G_\infty/K)) -(\rk G_\infty - \rk K_\infty)}{2}\right. + &\abs{S}(n-1), \\ &\left.\tfrac{(\dim (G_\infty/K)) +(\rk G_\infty - \rk K_\infty) }{2} + \abs{S}(n-1)\right]\end{align*}
\end{thm}
\begin{proof}
    By Theorem \ref{thm:S-arithMatsushima}, we have \[\HH^q(\Gamma;E) \cong \HH^q(\fg,K; E) \oplus \bigoplus_{\mathbbm{1}\ne\pi\in \hat{G}} \HH_d^{q-\sum_{\nu\in S_f}\rk(G_\nu)}(G_\infty; V_\pi\tensor E)^{m_{\pi,\Gamma}}.\] Since $\cG = \SL_n$, we have $\rk(G_\nu) = n-1$ for all $\nu\in S_f$. Theorem \ref{thm:Wallach} implies that if the summand corresponding to $\pi$ is nonzero, then $V_\pi$ is cuspidal. Therefore, Theorem \ref{thm:tempered} now implies that the cohomology groups of this term are concentrated in the stated range. The theorem follows.
\end{proof}

Theorem \ref{mainthm:automorphic} follows immediately by applying Proposition \ref{prop:gKCoh}, and precisely calculating the bounds. Corollary \ref{maincor:highdeg} follows from Corollary \ref{cor:gKCoh}.

\section{Representation Structure of Cohomology Groups}

In this section, we will show how to extract representation-theoretic data from the spectral sequence of Theorem \ref{thm:SpectralSequence}. Then, we show how to exploit the computation of the cohomology of $\SL_n(\cO_{F,S})$ to compute enough of this spectral sequence to extract Theorems \ref{mainthm:multiplicity} and \ref{mainthm:multiplicityCodimOne}. 

\subsection{Representation theory of $\SL_n(\bbF_q)$} 

We review some representation theory for $\SL_n(\bbF_q)$. The perspective here is originally due to Harish-Chandra, and mirrors some of the representation theory of Lie groups. See \cite[Chapter 47]{Bump} for an overview.

Let $P$ be a standard parabolic subgroup of $\SL_n(\bbF_q).$ We have a Levi decomposition:
\[U_P\to P\to L_P,\] where $L_P$ is of the form $S(\GL_{k_1}\times\dots\times \GL_{k_m})$ and $U_P$ is a $p$-group consisting of upper-triangular matrices. Given a representation $V$ of $L_P,$ we can produce a representation of $\SL_n$ by first extending $V$ to $P$ by the quotient map, then inducing from $P$ to $\SL_n$. 

\begin{defn}
    Let $W$ be an irreducible representation of $\SL_n(\bbF_q).$ We say that $W$ is \emph{parabolically induced} if there is a standard parabolic subgroup $P$ and irreducible representation $V$ of $L_P$ such that $W$ appears as a summand of $\Ind_P^{\SL_n} V.$ Otherwise, we say that $W$ is \emph{cuspidal}.
\end{defn}
Given an irreducible representation $W$ of $\SL_n(\bbF_q),$ it is possible to determine whether it is parabolically induced and discover the data $V,P.$
\begin{prop}\label{prop:Jacquet}
    Let $W$ be an irreducible representation of $\SL_n(\bbF_q),$ let $P$ be a standard parabolic subgroup, and let $V$ be an irreducible representation of the Levi factor $L_P.$ Then, the multiplicity of $W$ in $\Ind_P^{\SL_n}V$ is equal to the multiplicity of $V$ in the $L_P$-representation $(\Res_P^{\SL_n} W)_{U_P}.$ In particular, $W$ is cuspidal if and only if $(\Res_P^{\SL_n} W)_{U_P}=0,$ for all standard parabolics $P$.
\end{prop}
\begin{proof}
    We have \[\Hom_{\SL_n}(W, \Ind_P^{\SL_n} V) = \Hom_{P}(\Res_P^{\SL_n} W, V).\] Since $V$ is pulled back from $L_P$, the action of $U_P$ on $V$ is trivial. Therefore, every map to $V$ factors uniquely through the coinvariants, and we conclude that \[\Hom_{\SL_n}(W, \Ind_P^{\SL_n} V) =\Hom_{L_P}((\Res_P^{\SL_n} W)_{U_P}, V).\]
    By definition, $W$ is cuspidal if and only if, for every parabolic $P$ and every irreducible representation $V$ of $L_P$, these spaces of homomorphisms are all 0. Since $V$ is an arbitrary irreducible representation of $L_P$, Maschke's theorem implies that this occurs if and only if $(\Res_P^{\SL_n} W)_{U_P}=0$.
\end{proof}

The cohomology groups appearing on the $E_1$-page of the spectral sequence are all of the form $\HH^q(\Gamma_n(\fp);\bbC)^{P(\cO/\fp)}.$ We obtain the following corollary of Proposition \ref{prop:Jacquet}:
\begin{cor}\label{cor:parabolicInvariants}
    Let $V$ be any finite-dimensional representation of $\SL_n(\bbF_q).$ Let $P$ be a parabolic subgroup. Then, $\dim V^P$ is the sum of the multiplicities of irreducible subrepresentations of $V$ which appear in $\Ind_P^{\SL_n}\mathbbm{1}.$
\end{cor}
\begin{proof}
    Observe that $\dim V^P= \dim V_P$, and that $V_P = \left(V_{U_P}\right)_{L_P}.$ Therefore, $\dim V_P$ exactly counts the trivial $L_P$-subrepresentations of $V_{U_P}$. The statement now follows from Proposition \ref{prop:Jacquet}.
\end{proof}
\subsection{Proofs of Theorems \ref{mainthm:multiplicity} and \ref{mainthm:multiplicityCodimOne}}
Fix a number field $F$, and set $\nu = \nu_F(n).$
The following result is the main technical tool used to prove Theorems \ref{mainthm:multiplicity} and \ref{mainthm:multiplicityCodimOne}.
\begin{thm}\label{thm:interpretMult}
    In the spectral sequence of Theorem \ref{thm:SpectralSequence}, we have that $\dim E_2^{(n-1),q}$ is an upper bound on the multiplicity of the Steinberg representation $\St_n(\cO/\fp)$ in $\HH^q(\Gamma_n(\fp);\bbC).$
\end{thm}
\begin{proof}
    We have \begin{align*} E_2^{(n-1),q} &\cong \text{cok}\left( \bigoplus_{i\in \Sigma} \HH^{q}(\Gamma^{P_i}_n(p);\bbQ)\overset{d_1}{\longrightarrow} \HH^{q}(\Gamma^{B}_n(p);\bbQ)\right)\\ &\cong \text{cok}\left( \bigoplus_{i\in \Sigma} \HH^{q}(\Gamma_n(p);\bbQ)^{P_i}\overset{d_1}{\longrightarrow} \HH^{q}(\Gamma_n(p);\bbQ)^{B}\right)\end{align*}
    Let $\St_n(\cO/\fp)^m$ be the isotypic component of $\HH^{q}(\Gamma_n(p);\bbQ)$ corresponding to the Steinberg module. Then, using Corollary \ref{cor:parabolicInvariants}, we have $(\St_n(\cO/\fp)^m)^{P_i} = 0,$ and $(\St_n(\cO/\fp)^m)^{B} \cong \bbQ^m.$ By additivity of the invariants functors and our description of $d_1,$ it follows that $\bbQ^m$ projects to the cokernel. Therefore, $\dim E_2^{(n-1),q}\geq m$ as required.
\end{proof}

We prove Theorem \ref{mainthm:multiplicity}, which states that if $n$ is large enough that the cuspidal spectrum is bounded away from the rational cohomological dimension, then for any prime ideal $\fp,$ the multiplicity of the Steinberg representation in $\HH^\nu(\Gamma_n(\fp);\bbC)$ is one.

\begin{proof}[Proof of Theorem \ref{mainthm:multiplicity}]
    Recall that $S_\fp = \{\abs{-}_\fp\}\cup \Omega_\infty.$ By Corollary \ref{maincor:highdeg}, we have that $\HH^{\nu+n-1}(\SL_n(\cO_{F,S_\fp});\bbC)\cong\bbC.$ By Corollary \ref{cor:specSeqBound}, we have $E_\infty^{n-1,\nu} = E_2^{n-1,\nu} = \bbC.$ Theorem \ref{thm:interpretMult} now implies that the multiplicity of the Steinberg module in $\HH^\nu(\Gamma_n(\fp);\bbC)$ is at most one. By \cite[Theorem E]{S}, the multiplicity is at least one, and therefore it is equal to one.
\end{proof}

For the remainder of the paper, we will assume that $\cO,\fp$ are such that \[\HH^\nu(\Gamma_n(\fp);\bbC) \cong \St^{\cO^\times}_n(\cO/\fp).\] We use this to compute the upper edge of the spectral sequence.

\begin{thm}\label{thm:TopEdge}
    Assume that $\cO,\fp$ are such that $\HH^\nu(\Gamma_n(\fp);\bbC) \cong \St^{\cO^\times}_n(\cO/\fp).$ Then, if $P$ is a standard parabolic subgroup of $\SL_n,$ we have
    \[\HH^{\nu}(\Gamma_n^P(\fp);\bbC) \cong \begin{cases}
        \bbC & P = B_n \\ 0 &\text{otherwise.}
    \end{cases}\]
    In particular, $E_1^{p,\nu} = 0$ unless $p=n-1.$
\end{thm}
If $\cO =\bbZ$ and $P$ is the stabilizer of a line, Abdelnaim obtains a vanishing result for a range of primes increasing with $n$ \cite{Abdelnaim}.
\begin{proof}
The action of $\Gamma_n(\fp)$ on $\cT_n(\cO)$ is simplicial and does not invert edges. Therefore, there is an induced simplicial structure on $\cT_n(\cO)/\Gamma_n(\fp)$ whose $k$-simplices are orbits of simplices of $\cT_n(\cO).$ Since $C_{n-2}(\cT_n(\cO)) = \bbC[\SL_n(\cO)/B_n(\cO)],$ we have \[C_{n-2}(\Gamma_n(\fp)\backslash\cT_n(\cO)) \cong \bbC[\Gamma_n(\fp)\backslash\SL_n(\cO)/B_n(\cO)]\cong \bbC[\SL_n(\cO/\fp)/\text{im}(B_n(\cO))].\]
The group $B_n$ has a Levi decomposition
\[ U_n\to B_n\to T_n.\]
An elementary matrix of $U_n(\cO/\fp)$ has a lift in $U_n(\cO).$ Since elementary matrices generate $U_n$, an element of $B_n(\cO/\fp)$ lies in the image of $B_n(\cO)$ if and only if its image in $T_n(\cO/\fp)$ lies in the image of $T_n(\cO).$ 

We have \[C_{n-2}(\Gamma_n(\fp)\backslash\cT_n(\cO)) \cong \Ind_{\text{im}(B_n(\cO))}^{\SL_n(\cO/\fp)}\bbC \cong \Ind_{B_n(\cO)}^{\SL_n(\cO/\fp)} \Ind_{\text{im}(B_n(\cO))}^{B_n(\cO/\fp)}\bbC,\]
and \[\Ind_{\text{im}(B_n(\cO))}^{B_n(\cO/\fp)}\bbC = \bigoplus_{\substack{\chi:T_n(\bbF_p)\to \bbC^\times\\ \chi|_{\text{im}(T_n(\cO)) = \mathbbm{1}}}} V_\chi.\] Since $\Ind$ is additive, we finally obtain \[C_{n-2}(\Gamma_n(\fp)\backslash\cT_n(\cO)) \cong \bigoplus_{\substack{\chi:T_n(\bbF_p)\to \bbC^\times\\ \chi|_{\text{im}(T_n(\cO)) = \mathbbm{1}}}} \Ind_{B_n(\cO)}^{\SL_n(\cO/\fp)} V_\chi.\]

There is a natural simplicial map $\Gamma_n(\fp)\backslash\cT_n(\cO)\to \cT_n(\cO/\fp),$ induced from the functoriality of Tits complexes in \cite{S} and the universal property of the quotient. Tracing through the definitions, we find that the map $C_{n-2}(\Gamma_n(\fp)\backslash\cT_n(\cO))\to C_{n-2}(\cT_n(\cO/\fp))$ is the map \[ \bigoplus_{\substack{\chi:T_n(\bbF_p)\to \bbC^\times\\ \chi|_{\text{im}(T_n(\cO)) = \mathbbm{1}}}} \Ind_{B_n(\cO)}^{\SL_n(\cO/\fp)} V_\chi \to \Ind_{B_n(\cO)}^{\SL_n(\cO/\fp)} \bbC\] given by projection onto the $\chi=\mathbbm{1}$ summand. We have a diagram with exact rows \[\xymatrix{0\ar[r] & \St^{\cO^\times}_n(\cO/\fp)\ar[r] \ar[d] & \displaystyle\bigoplus_{\substack{\chi:T_n(\bbF_p)\to \bbC^\times\\ \chi|_{\text{im}(T_n(\cO)) = \mathbbm{1}}}} \Ind_{B_n(\cO)}^{\SL_n(\cO/\fp)} V_\chi \ar[d] \\ 0 \ar[r] & \St_n(\cO/\fp)\ar[r] &  \Ind_{B_n(\cO)}^{\SL_n(\cO/\fp)} \bbC.}\]

Let $P_\lambda$ be a standard parabolic subgroup of $\SL_n.$ By Corollary \ref{cor:parabolicInvariants}, the map \[ \left(\bigoplus_{\substack{\chi:T_n(\bbF_p)\to \bbC^\times\\ \chi|_{\text{im}(T_n(\cO)) = \mathbbm{1}}}} \Ind_{B_n(\cO)}^{\SL_n(\cO/\fp)} V_\chi\right)^{P_\lambda(\cO/\fp)} \to \left(\Ind_{B_n(\cO)}^{\SL_n(\cO/\fp)} \bbC\right)^{P_\lambda(\cO/\fp)}\] is an isomorphism. By exactness of the invariants functor, the diagram \[\xymatrix{0\ar[r] & \St^{\cO^\times}_n(\cO/\fp)^{P_\lambda(\cO/\fp)}\ar[r] \ar[d] & \displaystyle \left(\bigoplus_{\substack{\chi:T_n(\bbF_p)\to \bbC^\times\\ \chi|_{\text{im}(T_n(\cO)) = \mathbbm{1}}}} \Ind_{B_n(\cO)}^{\SL_n(\cO/\fp)} V_\chi\right)^{P_\lambda(\cO/\fp)} \ar[d]^{\wr} \\ 0 \ar[r] & \St_n(\cO/\fp)^{P_\lambda(\cO/\fp)}\ar[r] &  \left(\Ind_{B_n(\cO)}^{\SL_n(\cO/\fp)} \bbC\right)^{P_\lambda(\cO/\fp)}}\] has exact rows. A diagram chase shows that the map \[\St^{\cO^\times}_n(\cO/\fp)^{P_\lambda(\cO/\fp)}\to \St_n(\cO/\fp)^{P_\lambda(\cO/\fp)}\] is injective. 

Since \[\St_n(\cO/\fp)^{P_\lambda(\cO/\fp)}\cong\begin{cases}
    \bbC & P_\lambda = B_n \\ 0 & \text{otherwise,}
\end{cases}\] we conclude that $\St^{\cO^\times}_n(\cO/\fp)^{P_\lambda(\cO/\fp)} = 0$ unless $P_\lambda = B_n$. It follows that $E_1^{p,\nu} = 0$ unless $p = n-1.$ As in the proof of Theorem \ref{mainthm:multiplicity}, we obtain that $E_2^{n-1,\nu} = \bbC$. But since $E_1^{n-2,\nu} = 0,$ we must likewise have $E_1^{n-1,\nu} = \bbC$, proving that $\St^{\cO^\times}_n(\cO/\fp)^{B_n(\cO/\fp)} \cong \bbC$ in this case.
\end{proof}

We are now ready to prove Theorem \ref{mainthm:multiplicityCodimOne}, which states that if $\cO,\fp$ are such that $\HH^\nu(\Gamma_n(\fp);\bbC) \cong \St^{\cO^\times}_n(\cO/\fp),$ then the multiplicity of $\St_n(\cO/\fp)$ in $\HH^{\nu-1}(\Gamma_n(\fp);\bbC)$ is zero.
\begin{proof}[Proof of Theorem \ref{mainthm:multiplicityCodimOne}]
    By Corollary \ref{maincor:highdeg}, we have that $E_\infty^{n-2,\nu} = E_\infty^{n-1,\nu-1} = 0.$ By Corollary \ref{cor:specSeqBound}, we have that $E_\infty^{n-1,\nu-1}=E_3^{n-1,\nu-1} = 0$. There is a differential $d_2:E_2^{n-1,\nu-1}\to E_2^{n-3,\nu}$ whose kernel is $E_3^{n-1,\nu-1}=0$, and is therefore injective. Theorem \ref{thm:TopEdge} implies that $E^{n-2,\nu}_2 = 0$, so $E_2^{n-1,\nu-1} = 0.$ Theorem \ref{thm:interpretMult} now implies that the multiplicity of the Steinberg representation in $\HH^{\nu-1}(\Gamma_n(\fp);\bbC)$ is zero.
\end{proof}

\bibliographystyle{plain}
	\bibliography{bibliography}

\end{document}